\documentclass[10pt]{amsart}
\usepackage{epsfig}
\usepackage{graphicx}
\usepackage{amscd}
\usepackage{amsmath}
\usepackage{amsxtra}
\usepackage{amsfonts}
\usepackage{amssymb}

\newtheorem{theorem}{Theorem}[section]
\newtheorem{corollary}[theorem]{Corollary}
\newtheorem{lemma}[theorem]{Lemma}
\newtheorem{proposition}[theorem]{Proposition}

\newtheorem{conjecture}[theorem]{Conjecture}

\theoremstyle{definition}
\newtheorem{definition}[theorem]{Definition}

\theoremstyle{remark}

\renewcommand{\theclaim}{\textup{\theclaim}}

\newtheorem*{acknowledgements}{Acknowledgements}

\numberwithin{equation}{section}

\newcommand\restr[2]{{% we make the whole thing an ordinary symbol
  \left.\kern-\nulldelimiterspace % automatically resize the bar with \right
  #1 % the function
  \vphantom{\big|} % pretend it's a little taller at normal size
  \right|_{#2} % this is the delimiter
  }}

\def\openone%{\hbox{\upshape \small1\kern-3.3pt\normalsize1}}

{\mathchoice

{\hbox{\upshape \small1\kern-3.3pt\normalsize1}}

{\hbox{\upshape \small1\kern-3.3pt\normalsize1}}

{\hbox{\upshape \tiny1\kern-2.3pt\SMALL1}}

{\hbox{\upshape \Tiny1\kern-2pt\tiny1}}}

\makeatletter

\newbox\ipbox

\newcommand{\ip}[2]{\left\langle #1\, , \,#2\right\rangle}
\newcommand{\diracb}[1]{\left\langle #1\mathrel{\mathchoice

{\setbox\ipbox=\hbox{$\displaystyle \left\langle\mathstrut
#1\right.$}

\vrule height\ht\ipbox width0.25pt depth\dp\ipbox}

{\setbox\ipbox=\hbox{$\textstyle \left\langle\mathstrut
#1\right.$}

\vrule height\ht\ipbox width0.25pt depth\dp\ipbox}

{\setbox\ipbox=\hbox{$\scriptstyle \left\langle\mathstrut
#1\right.$}

\vrule height\ht\ipbox width0.25pt depth\dp\ipbox}

{\setbox\ipbox=\hbox{$\scriptscriptstyle \left\langle\mathstrut
#1\right.$}

\vrule height\ht\ipbox width0.25pt depth\dp\ipbox}

}\right. }

\newcommand{\dirack}[1]{\left. \mathrel{\mathchoice

{\setbox\ipbox=\hbox{$\displaystyle \left.\mathstrut
#1\right\rangle$}

\vrule height\ht\ipbox width0.25pt depth\dp\ipbox}

{\setbox\ipbox=\hbox{$\textstyle \left.\mathstrut
#1\right\rangle$}

\vrule height\ht\ipbox width0.25pt depth\dp\ipbox}

{\setbox\ipbox=\hbox{$\scriptstyle \left.\mathstrut
#1\right\rangle$}

\vrule height\ht\ipbox width0.25pt depth\dp\ipbox}

{\setbox\ipbox=\hbox{$\scriptscriptstyle \left.\mathstrut
#1\right\rangle$}

\vrule height\ht\ipbox width0.25pt depth\dp\ipbox}

} #1\right\rangle}

\newcommand{\cj}[1]{\overline{#1}}

\newcommand{\bz}{\mathbb{Z}}

\newcommand{\br}{\mathbb{R}}
\newcommand{\bc}{\mathbb{C}}

\newcommand{\bn}{\mathbb{N}}

\def\blfootnote{\xdef\@thefnmark{}\@footnotetext}

\input xy
\xyoption{all}
\usepackage{amssymb}

\def\-{^{-1}}

\def\ty{\emptyset}

\begin{document}

\title[On the universal tiling conjecture in dimension one]{On the universal tiling conjecture in dimension one}
\author{Dorin Ervin Dutkay}
\blfootnote{}
\address{[Dorin Ervin Dutkay] University of Central Florida\\
    Department of Mathematics\\
    4000 Central Florida Blvd.\\
    P.O. Box 161364\\
    Orlando, FL 32816-1364\\
U.S.A.\\} \email{Dorin.Dutkay@ucf.edu}

\author{Palle E.T. Jorgensen}
\address{[Palle E.T. Jorgensen]University of Iowa\\
Department of Mathematics\\
14 MacLean Hall\\
Iowa City, IA 52242-1419\\}\email{jorgen@math.uiowa.edu}

\thanks{}

\begin{abstract} We show that the spectral-tile implication in the Fuglede conjecture in dimension 1 is equivalent to a Universal Tiling Conjecture and also to similar forms of the same implication for some simpler sets, such as unions of intervals with rational or integer endpoints.
\end{abstract}
\maketitle \tableofcontents

\section{Introduction}

\begin{definition}\label{def1.1}
For $\lambda\in\br$ we denote by 
$$e_\lambda(x):=e^{2\pi i \lambda x},\quad(x\in\br)$$
Let $\Omega$ be Lebesgue measurable subset of $\br$ with finite Lebesgue measure. We say that $\Omega$ is {\it spectral} if there exists a subset $\Lambda$ of $\br$ such that $\{e_\lambda :\lambda\in\Lambda\}$ is an orthogonal basis for $L^2(\Omega)$. In this case $\Lambda$ is called a {\it spectrum} for $\Omega$.

We say that $\Omega$ {\it tiles $\br$ by translations} if there exists a subset $\mathcal T$ of $\br$ such that $\{\Omega+t : t\in\mathcal T\}$ is a partition of $\br$, up to Lebesgue measure zero.

\end{definition}

Fuglede's conjecture was stated for arbitrary finite dimension in \cite{Fug74}. It asserts that the tiling and the spectral properties are equivalent. Tao \cite{Tao04} disproved one direction in the Fuglede conjecture in dimensions 5 or higher: there exists a union of cubes which is spectral but does not tile. Later, Tao's counterexample was improved to disprove both directions in Fuglede's conjecture for dimensions 3 or higher \cite{KM06,FaMaMo06}. In the cases of dimensions 1 and 2, both directions are still open. 

We state here Fuglede's conjecture in dimension 1:
\begin{conjecture}\label{co1.2}\cite{Fug74}
A subset $\Omega$ of $\br$ of finite Lebesgue measure is spectral if and only if it tiles $\br$ by translations.
\end{conjecture}

We focus on the spectral-tile implication in the Fuglede conjecture and present some equivalent statements. One of the main ingredients that we will use is the fact that any spectrum is {\it periodic} (see \cite{BoMa11,IoKo12}).

 There are good reasons for our focus on the cases of the conjectures in one dimension. One reason is periodicity (see the next definition): it is known, in 1D, that the possible sets $\Lambda$ serving as candidates for spectra, in the sense of Fuglede's conjecture (Conjecture \ref{co1.2}), must be periodic.  A second reason lies in the difference, from 1D to 2D, in the possibilities for geometric configurations of translation sets.  

The Universal Tiling Conjecture (Conjecture \ref{con4.8}) suggests a reduction of the implication from spectrum to tile in Fuglede's conjecture, to a consideration of finite subsets of  $\bz$. Hence computations for the problems in 1D are arithmetic in nature, as opposed to geometric; and connections to classical Fourier series may therefore be more direct in 1D.

If $\Omega$ is spectral then any spectrum $\Lambda$ is {\it periodic} with some period $p\neq0$, i.e., $\Lambda+p=\Lambda$, and $p$ is an integer multiple of $\frac{1}{|\Omega|}$. We call $p$ a period for $\Lambda$. If $p=\frac{k(p)}{|\Omega|}$ with $k(p)\in\bn$, then $\Lambda$ has the form

\begin{equation}
\Lambda=\{\lambda_0,\dots,\lambda_{k(p)-1}\}+p\bz,
\label{eq3.8.1}
\end{equation}  
with $\lambda_0,\dots,\lambda_{k(p)-1}\in[0,p)$, see \cite{BoMa11,IoKo12}. The reason that there are $k(p)$ elements of $\Lambda$ in the interval $[0,p)$ can be seen also from the fact that the Beurling density of a spectrum $\Lambda$ has to be ${|\Omega|}$, see \cite{Lan67a}.

These assertions follow from \cite{IoKo12}. According to \cite{IoKo12}, if $\Omega$ has Lebesgue measure 1 and is spectral with spectrum $\Lambda$, with $0\in\Lambda$, then $\Lambda$ is periodic, the period $p$ is an integer and $\Lambda$ has the form
\begin{equation}
\Lambda=\{\lambda_0=0,\lambda_1,\dots,\lambda_{p-1}\}+p\bz,
\label{eq4.1}
\end{equation}
where $\lambda_i$ in $[0,p)$ are some distinct real numbers.

\begin{definition}\label{def1.5}
Let $A$ be a finite subset of $\br$. We say that $A$ is {\it spectral} if there exists a finite set $\Lambda$ in $\br$ such that $\{e_\lambda : \lambda\in\Lambda\}$ is a Hilbert space for $L^2(\delta_A)$ where 
$\delta_A$ is the atomic measure $\delta_A:=\sum_{a\in A}\delta_a$ and $\delta_a$ is the Dirac measure at $a$. We call $\Lambda$ a {\it spectrum} for $A$. 

\end{definition}
We formulate the following "Universal Tiling Conjecture" for a fixed number $p\in\bn$:
\begin{conjecture}\label{con4.8}{\bf[UTC($p$)]}
Let $p\in\bn$. Let $\Gamma:=\{\lambda_0=0,\lambda_1,\dots,\lambda_{p-1}\}$ be a subset of $\br$ with $p$ elements. Assume $\Gamma$ has a spectrum of the form $\frac1pA$ with $A\subset \bz$. Then for every finite family $A_1,A_2,\dots,A_n$ of subsets of $\bz$ such that $\frac1pA_i$ is a spectrum for $\Gamma$ for all $i$, there exists a common tiling subset $\mathcal T$ of $\bz$ such that the set $A_i$ tiles $\bz$ by $\mathcal T$ for all $i\in\{1,\dots,n\}$.
\end{conjecture}

One of the main results in this paper shows the equivalence of the spectral-tile implication in Fuglede's conjecture and the Universal Tiling Conjecture. 
\begin{theorem}\label{th4.8}
The following affirmations are equivalent.
\begin{enumerate}
	\item The Universal Tiling Conjecture is true for all $p\in\bn$.
	\item Every bounded Lebesgue measurable spectral set tiles by translations. 
\end{enumerate}
Moreover, if these statements are true and if $\Omega$, $|\Omega|=1$, is a bounded Lebesgue measurable set which has a spectrum with period $p$, then $\Omega$ tiles by a subset $\mathcal T$ of $\frac1p\bz$. 
\end{theorem}

The term ``universal tiling'' appears also in \cite{PeWa01} for a special class of tiles of $\br$, namely those that tile $\br^+$. On the dual side, the Universal Spectrum Conjecture was introduced in \cite{LaWa97} where it is proved that some sets $\Omega$ which tile by some special tiling set $\mathcal T$ have a spectrum $\Lambda$ which depends only on $\mathcal T$. In \cite{FaMaMo06} it is proved that the Universal Spectrum Conjecture is equiv-
alent to the tile-spectral implication in Fuglede's conjecture, in the case of
finite abelian groups. The notion of universal tiling complements is also in
troduced in \cite{FaMaMo06}, and it is remarked (see Remark 2 in \cite{FaMaMo06}) that
the spectral-tile implication in Fuglede's conjecture is equivalent to a universal tiling conjecture, again for finite abelian groups. In dimension 1 this
means that for cyclic groups the spectral-tile implication is equivalent to all
spectral sets possessing a universal tiling complement. In Theorem \ref{th4.8} we
prove this result in full generality, for any bounded Lebesgue measurable sets $\Omega$ as in the
original setting in the Fuglede conjecture.

The second main result in this paper shows that the spectral-tile implication in the Fuglede conjecture is equivalent to some formulations of this implication for some special classes of sets $\Omega$: unions of intervals with rational or integer endpoints.

\begin{theorem}\label{th4.13}
The following affirmations are equivalent:
\begin{enumerate}
	\item For every finite union of intervals with rational endpoints $\Omega=\cup_{i=1}^n(\alpha_i,\beta_i)$ with $|\Omega|=1$, if $\Omega$ has a spectrum $\Lambda$ with period $p$, then $\Omega$ tiles $\br$ by a subset $\mathcal T$ of $\frac1p\bz$.
	\item For every finite union of intervals with integer endpoints $\Omega=\cup_{i=1}^n(\alpha_i,\beta_i)$, $|\Omega|=N$, if $\Omega$ has a spectrum $\Lambda$ with minimal period $\frac rN$, $r\in\bz$, then $\frac Nr$ is an integer (see Corollary \ref{cor3.9}) and $\Omega$ tiles $\br$ 
with a subset $\mathcal T$ of $\frac Nr\bz$. 
	\item Every bounded Lebesgue measurable spectral set tiles by translations. 
\end{enumerate}
\end{theorem}

\section{Analysis of spectral sets}

We begin with a few lemmas and propositions that exploit the periodicity of the spectrum to give some information about the structure of $\Omega$. We show in Proposition \ref{pr4.7} that if a set $\Omega$ with $|\Omega|=1$, has a spectrum $\Lambda$ of period $p$, $\Lambda=\{0=\lambda_0,\dots,\lambda_{p-1}\}+p\bz$, then for a.e. $x\in\Omega$ the set $\Omega_x=\{k\in\bz: x+\frac kp\in\Omega\}$ has exactly $p$ elements and the sets $\Omega_x$ have a common spectrum $\frac{1}p\{\lambda_0,\dots,\lambda_{p-1}\}$. This will enable us to use the Universal Tiling Conjecture to show that, in this case, the sets $\Omega_x$ have a common tiling set in $\bz$ and therefore $\Omega$ is a tile for $\br$ .

\begin{proposition}\label{pr3.9}
If $\Omega=\cup_{i=1}^n(\alpha_i,\beta_i)$, $\alpha_1<\beta_1<\alpha_2<\beta_2<\dots<\alpha_n<\beta_n$ is spectral, $p\in\bn$ and $\alpha_i,\beta_i\in\frac1p\bz$ for all $i=1,\dots,n$, then any spectrum $\Lambda$ for $\Omega$ has $p$ as a period. 
\end{proposition}

\begin{proof}
Let $\Lambda$ be a spectrum for $\Omega$. Take $\lambda\in\Lambda$ and assume $\lambda+p\not\in\Lambda$. We have, for $\lambda'\in\Lambda$:
$$\ip{e_{\lambda+p}}{e_{\lambda'}}_{L^2(\Omega)}=\frac{1}{|\Omega|}\sum_{i=1}^n\frac{1}{2\pi i(\lambda+p-\lambda')}(e^{2\pi i (\lambda+p-\lambda')\beta_i}-e^{2\pi i(\lambda+p-\lambda')\alpha_i})$$$$=\frac{1}{2|\Omega|\pi i(\lambda+p-\lambda')}\sum_{i=1}^n(e^{2\pi i (\lambda-\lambda')\beta_i}-e^{2\pi i(\lambda-\lambda')\alpha_i})=\frac1{|\Omega|}\frac{\lambda-\lambda'}{\lambda+p-\lambda'}\ip{e_\lambda}{e_{\lambda'}}_{L^2(\Omega)}=0$$
for both cases $\lambda\neq\lambda'$ and $\lambda=\lambda'$. But this would contradict the completeness of $\{e_\lambda : \lambda\in\Lambda\}$. 
\end{proof}

\begin{corollary}\label{cor3.9}
If $\Omega=\cup_{i=1}^n(\alpha_i,\beta_i)$, $\alpha_1<\beta_1<\alpha_2<\beta_2<\dots<\alpha_n<\beta_n$ is spectral and $\alpha_i,\beta_i\in\bz$ for all $i=1,\dots,n$ and if a spectrum $\Lambda$ has minimal period $\frac{k}{|\Omega|}$ then $k$ divides $|\Omega|$. 
\end{corollary}

\begin{proof}
$|\Omega|$ is an integer $N$. Then $\frac{1}{N}\Omega$ has measure 1, endpoints in $\frac1N\bz$ and spectrum $N\Lambda$ with period $k$. By Proposition \ref{pr3.9}, $k$ has to divide $N$. 
\end{proof}
The next proposition can be found in e.g. \cite{BoMa11,LaWa97,Ped96,IoKo12}. We include the proof for the benefit of the reader.
\begin{proposition}\label{pr4.15}
Let $\Omega$ be a bounded Lebesgue measurable set of measure 1. Assume that $\Omega$ is spectral with spectrum $\Lambda$, $0\in\Lambda$, which has period $p$. Then $\Omega$ is a $p$-tile of $\br$ by $\frac{1}{p}\bz$-translations, i.e., for almost every $x\in\br$, there exist exactly $p$ integers $j_1,\dots, j_{p}$ such that $x$ is in $\Omega+\frac{j_i}{p}$, $i=1,\dots,p$.

Also, for a.e. $x$ in $\Omega$, there are exactly $p$ integers $j_1,\dots,j_{p}$ such that $x+\frac{j_i}{p}$ is in $\Omega$ for all $i=1,\dots,p$.  

\end{proposition}

\begin{proof}
The statements follow if we prove that
\begin{equation}
\frac{1}{p}\sum_{j\in\bz}\chi_\Omega\left(x+\frac{j}{p}\right)=1\mbox{ a.e. on $[0,\frac1p]$}.
\label{eq4.15.1}
\end{equation}
By assumption, $\Lambda$ contains $0$ so $p\bz\subset\Lambda$. 
We compute the Fourier coefficients: for $l\in\bz$, 
$$p\int_0^{\frac1p}\frac{1}{p}\sum_{j\in\bz}\chi_{\Omega}\left(x+\frac jp\right)e_{-lp}(x)\,dx=\sum_{j\in\bz}\int_{\frac jp}^{\frac{j+1}{p}}\chi_{\Omega}(x)e_{-lp}(x)\,dx=\int_{\br}\chi_{\Omega}(x)e_{-lp}(x)\,dx=\delta_0.$$
This implies \eqref{eq4.15.1}.

\end{proof}

\begin{definition}\label{def4.16}
For $\varphi\in L^\infty(\br)$, define the multiplication operator $M(\varphi)$ on $L^2(\br)$ by $M(\varphi)f=\varphi f$, $f\in L^2(\Omega)$.
\end{definition}

\begin{lemma}\label{lem4.16}
Let $\Omega$ be a bounded Lebesgue measurable set of measure 1. Assume that $\Omega$ is spectral with spectrum $\Lambda$, $0\in\Lambda$ which has period $p$. Let $P(p\bz)$ be the orthogonal projection in $L^2(\Omega)$ onto the closed subspace spanned by $\{e_{kp} : k\in\bz\}$. Then, for $f\in L^2(\Omega)$,
\begin{equation}
(P(p\bz)f)(x)=\frac{1}{p}\sum_{j\in\bz}f\left(x+\frac jp\right),\mbox{ for a.e. $x\in\Omega$}.
\label{eq4.16.1}
\end{equation}
(We define $f$ to be zero outside $\Omega$)

\end{lemma}

\begin{proof}

Note first that the function on the right side of \eqref{eq4.16.1} has period $\frac1p$ and therefore it is in the $L^2$-span of the functions $e_{pl}$, $l\in\bz$. 
We have, for $f\in L^2(\Omega)$, $l\in\bz$: 
$$\int_{\Omega}\frac{1}{p}\sum_{j\in\bz}f\left(x+\frac jp\right)e_{-lp}(x)\,dx=\frac{1}{p}\sum_{j\in\bz}\int_{\br}\chi_\Omega\left(x-\frac jp\right)f(x)e_{-lp}(x)\,dx$$$$=\mbox{(with Proposition \ref{pr4.15})}=\int_\Omega f(x)e_{-lp}(x)\,dx=\ip{f}{e_{lp}}_{L^2(\Omega)}.$$

Therefore, since $e_{lp}$ form an orthogonal basis for their span, equation \eqref{eq4.16.1} follows.

\end{proof}

\begin{proposition}\label{pr4.17}
Let $\Omega$ be a bounded Lebesgue measurable set of measure 1. Assume that $\Omega$ is spectral with spectrum $\Lambda$, which has period $p$ and assume $0\in\Lambda$. Let $\Lambda=\{\lambda_0=0,\lambda_1,\dots,\lambda_{p-1}\}+p\bz$ with $\lambda_i\in[0,p)$, $i=0,\dots,p-1$. Then the projection $P(\lambda_i+p\bz)$ onto the span of $\{e_{\lambda_i+kp}: k\in\bz\}$ has the following formula: for $f\in L^2(\Omega)$,
\begin{equation}
(P(\lambda_i+p\bz)f)(x)=e_{\lambda_i}(x)\frac{1}{p}\sum_{j\in\bz}f\left(x+\frac jp\right)e_{-\lambda_i}\left(x+\frac jp\right),\mbox{ for a.e. $x\in\Omega$}.
\label{eq4.17.1}
\end{equation}
\end{proposition}

\begin{proof}
A simple check shows that for $\lambda\in\Lambda$ we have that the rank one projections $P(\lambda)$ onto $e_\lambda$ are related by the following formula:
\begin{equation}
P(\lambda)=M(e_\lambda)P(0)M(e_{-\lambda}).
\label{eq4.17.2}
\end{equation}
Then, we obtain
\begin{equation}
P(\lambda_i+p\bz)=M(e_{\lambda_i})P(p\bz)M(e_{-\lambda_i}).
\label{eq4.17.3}
\end{equation}
Equation \eqref{eq4.17.1} follows from \eqref{eq4.17.3} and \eqref{eq4.16.1}.
\end{proof}

\begin{proposition}\label{pr4.19}
Let $\Lambda=\{\lambda_0=0,\lambda_1,\dots,\lambda_{p-1}\}+p\bz$ be as in Proposition \ref{pr4.17}. For $x\in\br$, let $\Omega_x=\{j\in\bz : x+\frac jp\in\Omega\}$. Then $|\Omega_x|=p$ for a.e. $x\in \br$ and, for $i,i'=0,\dots,p-1$:
\begin{equation}
\frac{1}{p}\sum_{j\in\Omega_x}e^{2\pi i(\lambda_{i'}-\lambda_{i})\frac jp}=\delta_{ii'}\mbox{ for a.e. $x\in\Omega$}.
\label{eq4.19.1}
\end{equation}
In other words , the set $\{\lambda_0,\dots,\lambda_{p-1}\}$ is spectral and, for a.e. $x\in\br$, $\frac1{p}\Omega_x$ is a spectrum for it.

\end{proposition}

\begin{proof}
The first statement is contained in Proposition \ref{pr4.15}. Equation \eqref{eq4.19.1} follows from Proposition \ref{pr4.17} and the fact that $P(\lambda_i+p\bz)e_{\lambda_{i'}}=0$ for $i\neq i'$, because $\Lambda$ is a spectrum.
\end{proof}

\begin{proposition}\label{pr4.7}
Let $\Omega$ be a bounded Lebesgue measurable set of measure 1. Assume that $\Omega$ is spectral with spectrum $\Lambda$ as in Proposition \ref{pr4.17} and, for $x\in\Omega$ define $\Omega_x\subset\bz$, $0\in\Omega_x$ as in Proposition \ref{pr4.19}. The sets $\Omega_x$ have a common spectrum $\frac1p\{\lambda_0=0,\lambda_1,\dots,\lambda_{p-1}\}$ for a.e. $x\in\Omega$. Assume in addition that the sets $\Omega_x$ tile $\bz$ by a common tiling set $\mathcal T\subset \bz$. 
Then $\Omega$ tiles $\br$ by $\frac1p\mathcal T$. 

\end{proposition}

\begin{proof}
The fact that $\frac1p\Omega_x$ have the common spectrum $\{\lambda_i\}$ is contained in Proposition \ref{pr4.19}. We focus on the tiling property. We can assume $0\in\mathcal T$. 
We know from Proposition \ref{pr4.15} that $\Omega$ $p$-tiles $\br$ by $\frac1p\bz$. 
Take $y\in\br$. Then there exists $k\in\bz$ and $x\in\Omega$ such that $x+\frac kp=y$. Since $\Omega_x$ tiles $\bz$ with $\mathcal T$, there exist $j\in\Omega_x$ and $t\in\mathcal T$ such that $k=j+t$. Then 
$x+\frac jp\in\Omega$ and $x+\frac jp+\frac tp=y$. So $\cup_{t\in\mathcal T}(\Omega+\frac tp)$ covers $\br$. 

Now assume $y=x_1+\frac{t_1}p=x_2+\frac{t_2}p$ with $x_1,x_2\in\Omega$ and $t_1,t_2\in\mathcal T$. Then $x_1-x_2\in\frac1p\bz$ and therefore, there exists $j\in\Omega_{x_1}$ such that $x_2=x_1+\frac jp$. Also $0\in\Omega_{x_1}$ since $x_1\in\Omega$. Then 
$x_1+\frac jp+\frac{t_2}p=x_1+\frac{t_1}p$ so $\frac jp+\frac{t_2}p=0+\frac{t_1}p$. But since $\Omega_{x_1}$ tiles with $\mathcal T$, this implies $j=0$ and $t_1=t_2$, so $x_1=x_2$. Thus the sets $\Omega+t$ are mutually disjoint.
\end{proof}

\begin{proposition}\label{pr4.8}
Let $\Omega$ be a bounded Lebesgue measurable subset of $\br$ with $|\Omega|=1$. Let $p\in\bn$. Suppose $\Omega$ $p$-tiles $\br$ by $\frac1p\bz$. Then, for a.e. $x\in\br$ the set 
\begin{equation}
\Omega_x:=\left\{k\in\bz : x+\frac kp\in\Omega\right\}
\label{eq4.8.1}
\end{equation}
has exactly $p$ elements 
\begin{equation}
\Omega_x=\{k_0(x)<k_1(x)<\dots<k_{p-1}(x)\}.
\label{eq4.8.2}
\end{equation}
For almost every $x\in\Omega$ there exist unique $y\in [0,\frac1p)$ and $i\in\{0,\dots,p-1\}$ such that $y+\frac{k_i(y)}p=x$. 

The functions $k_i$ have the following property
\begin{equation}
k_i(x+\frac1p)=k_i(x)-1,\quad(x\in\br,i=0,\dots,p-1).
\label{eq4.8.2.1}
\end{equation}

Consider the space of $\frac1p$-periodic vector valued functions $L^2([0,\frac1p),\bc^p)$. 
The operator $W:L^2(\Omega)\rightarrow L^2([0,\frac1p),\bc^p)$ defined by
\begin{equation}
(Wf)(x)=\begin{pmatrix}f\left(y+\frac{k_{0}(y)}p\right)\\
\vdots\\
f\left(y+\frac{k_{p-1}(y)}p\right)
\end{pmatrix},\quad(y\in[0,\frac1p),f\in L^2(\Omega)),
\label{eq4.8.3}
\end{equation}
is an isometric isomorphism with inverse
\begin{equation}
W^{-1}\begin{pmatrix}
f_0\\
\vdots\\
f_{p-1}\end{pmatrix}(x)=f_i(y),\mbox{ if }x=y+\frac{k_i(y)}p,\mbox{ with }y\in[0,\frac1p), i\in\{0,\dots,p-1\}.
\label{eq4.8.4}
\end{equation}

A set $\Lambda$ of the form $\Lambda=\{0=\lambda_0,\lambda_1,\dots,\lambda_{p-1}\}+p\bz$ is a spectrum for $\Omega$ if and only if $\{\lambda_0,\dots,\lambda_{p-1}\}$ is a spectrum for $\frac1p\Omega_x$ for a.e. $x\in[0,\frac1p)$. 

The exponential functions are mapped by $W$ as follows:
\begin{equation}
(We_{\lambda_i+np})(y)=e_{\lambda_i+np}(y)\begin{pmatrix}e_{\lambda_i}(\frac{k_0(y)}p)\\
\vdots\\
e_{\lambda_i}(\frac{k_{p-1}(y)}p)
\end{pmatrix}=:F_{i,n}(y),\quad (i=0,\dots,p-1, n\in\bz,y\in[0,\frac1p)).
\label{eq4.8.5}
\end{equation}

%
%For $x$ in $\br$ define the $p\times p$ unitary matrix $\M_x$ which has column vectors 
%$$v_i(x):=\frac{1}{\sqrt{p}}(e_{\lambda_i}(\frac{k_0(x)}p),e_{\lambda_i}(\frac{k_1(x)}p),\dots,e_{\lambda_i}(\frac{k_{p-1}(x)}p))^t,\quad i=0,\dots,p-1.$$
%
%
%Let $U_\Lambda$ be the group of local translations on $\Omega$ associated to a spectrum $\Lambda$. 
%Consider the one-parameter unitary group $\U_p$ on $L^2([0,\frac1p),\bc^p)$ defined by 
%\begin{equation}
%(\U_p(t)F)(x)=\M_x\M_{x+t}^*F(x+t),\quad (x,t\in\br, F\in L^2([0,\frac1p),\bc^p)).
%\label{eq4.8.6}
%\end{equation}
%Then $W$ intertwines $U_\Lambda$ and $\U_p$:
%\begin{equation}
%WU_\Lambda(t)=\U_p(t)W.
%\label{eq4.8.7}
%\end{equation}

\end{proposition}

\begin{proof}
The first statement follows from the fact that $\Omega$ $p$-tiles $\br$ by $\frac1p\bz$. The second statement follows from this and the fact that $[0,\frac1p)$ tiles $\br$ by $\frac1p\bz$.

To check that $W$ and $W^{-1}$, as defined, are inverse to each other requires just a simple computation. We verify that $W$ is isometric. 

For a subset $S$ of $\bz$ with $|S|=p$ define 
$$A_S:=\{x\in[0,\frac1p) : \Omega_x=S\}.$$
To see that $A_S$ is measurable, note first that the maps $k_i:\br\rightarrow\bz$ are measurable; this can be proved by induction on $i$, first we have $\{x: k_0(x)=m\}=\{x: x+\frac mp\in\Omega, x+\frac np\not\in\Omega\mbox{ for all }n<m\}$. Then $\{x: k_1(x)=m\}=\{x: k_0(x)<m, x+\frac mp\in\Omega, x+\frac np\not\in\Omega\mbox{ for all } n=k_0(x)+1,\dots, m-1\}$ etc. For a set $S=\{s_0<\dots<s_{p-1}\}$, the set $A_S$ is given by $\{x: k_0(x)=s_0,\dots, k_{p-1}(x)=s_{p-1}\}$ hence it is measurable.

Note that, since $\Omega$ is bounded, $A_S=\ty$ for all but finitely many sets $S$. Also we have the following partition of $\Omega$. 
$$\bigcup_{|S|=p}(A_S+\frac1p S)=\Omega.$$

Take $f\in L^2(\Omega)$.  We have 
$$\|Wf\|_{L^2([0,\frac1p),\bc^p)}^2=\int_0^{\frac1p}\sum_{j=0}^{p-1}\left|f\left(x+\frac{k_j(x)}p\right)\right|^2\,dx=
\sum_{|S|=p}\int_{A_S}\sum_{j=0}^{p-1}\left|f\left(x+\frac{k_j(x)}p\right)\right|^2\,dx$$
$$=\sum_{|S|=p}\int_{A_S}\sum_{s\in S}\left|f\left(x+\frac{s}p\right)\right|^2\,dx=\sum_{|S|=p}\sum_{s\in S}\int_{A_S+\frac sp}|f(x)|^2\,dx=\int_{\Omega}|f(x)|^2\,dx.$$

Equation \eqref{eq4.8.5} requires just a simple computation. 

If $\Lambda$ is a spectrum for $\Omega$, then we saw in Proposition \ref{pr4.19} that $\frac1p\Omega_x$ has spectrum $\{\lambda_0,\dots,\lambda_{p-1}\}$ for a.e. $x\in\br$.

For the converse, if $\{\lambda_0,\dots,\lambda_{p-1}\}$ is a spectrum for $\frac1p\Omega_x$ for a.e. $x\in[0,\frac1p)$, then for a.e. $x\in [0,\frac1p)$,the vectors $v_i(x)=\frac1{\sqrt{p}}(e_{\lambda_i}(\frac{k_0(x)}p),\dots,e_{\lambda_i}(\frac{k_{p-1}(x)}p))^t$, $i=0,\dots,p-1$, form an orthonormal basis for $\bc^p$. Then the functions $F_{i,n}$ in \eqref{eq4.8.5} form an orthonormal system for $L^2([0,\frac1p),\bc^p)$ as can be seen by a short computation. To see that the functions $F_{i,n}$ span the entire Hilbert space, take $H$ in $L^2([0,\frac1p),\bc^p)$ such that $H\perp F_{i,n}$ for all $i=0,\dots,p-1$, $n\in\bz$. Then 
$$0=\int_0^{\frac1p}e_{\lambda_i+np}(x)\ip{H(x)}{v_i(x)}_{\bc^p}\,dx.$$
Since the functions $e_{np}$ are complete in $L^2[0,\frac1p)$ it follows that $e_{\lambda_i}(x)\ip{H(x)}{v_i(x)}_{\bc^p}=0$ for a.e. $x\in[0,\frac1p)$. So $\ip{H(x)}{v_i(x)}=0$ for a.e. $x\in[0,\frac1p)$ and all $i=0,\dots,p-1$. Then $H(x)=0$ for a.e. $x\in[0,\frac1p)$. 

%
%Next, we check that $\U_p$ is well defined, so the function in \eqref{eq4.8.6} is $\frac1p$-periodic. We have 
%
%$$\M_{x+\frac{1}p}=\frac{1}{\sqrt p}\left(e_{\lambda_j}(\frac{k_i(x+\frac1p)}p)\right)_{i,j=0,\dots,p-1}=\frac{1}{\sqrt p} \left(e_{\lambda_j}(\frac{k_i(x)-1}{p})\right)_{i,j},$$
%therefore $\M_{x+\frac1p}=\M_xD_\lambda(\frac1p)^*$, where $D_\lambda(\frac{1}{p})$ is the diagonal matrix with entries $e_{\lambda_i}(\frac1p)$.
%
%
%
%
%We have, for $x,t\in\br$, $F\in L^2([0,\frac1p),\bc^p)$:
%$$\M_{x+\frac1p}\M_{x+t+\frac{1}p}^*F(x+t+\frac1p)=\M_xD_\lambda(\frac1p)^* D_\lambda(\frac1p)\M_{x+t}^*F(x+t)=\M_x\M_{x+t}^*F(x+t).$$
%
%The fact that the matrices $\M_x$ and $\M_{x+t}$ are unitary implies that $\U_p(t)$ is unitary. 
%
%
%To obtain \eqref{eq4.8.7}, it is enough to verify it on the basis $e_{\lambda_i+np}$ and that is equivalent to:
%\begin{equation}
%\U_p(t)F_{i,n}=e_{\lambda_i+np}(t)F_{i,n},\quad(i=0,\dots,p-1,n\in\bz).
%\label{eq4.8.8}
%\end{equation}
%
%
%Note first that $\M_{x}^*v_i(x)=\delta_i$ for all $x\in\br$, $i=0,\dots,p-1$. We have 
%$$(\U_p(t)F_{i,n})(x)=\M_x\M_{x+t}^*F_{i,n}(x+t)=\sqrt{p}e_{\lambda_i+np}(x+t)\M_x\M_{x+t}^*v_{i}(x+t)=\sqrt{p}e_{\lambda_i+np}(x+t)\M_x\delta_i$$$$=\sqrt{p}e_{\lambda_i+np}(t)e_{\lambda_i+np}(x)v_{i}(x)=e_{\lambda_i+np}(t)F_{i,n}(x).$$

\end{proof}

We can prove now our main results.

\begin{proof}[Proof of Theorem \ref{th4.8}]
Assume (i). Let $\Omega$ be a bounded Lebesgue measurable spectral set. By scaling we can assume $|\Omega|=1$. Then $\Omega$ has a spectrum of period $p$. The sets $\Omega_x$ in Proposition \ref{pr4.7} are obviously bounded by a common number $K$, since $\Omega$ is bounded, so they form a finite family. Then just apply the UTC($p$) conjecture to the sets $\Omega_x$ and use Proposition \ref{pr4.7}. We get also that $\Omega$ tiles by a subset of $\frac1p\bz$.

For the converse, assume (ii) holds. Let $p\in\bn$. Take $\Gamma=\{\lambda_0=0,\dots,\lambda_{p-1}\}$ which has a spectrum of the form $\frac1pA$ with $A\subset\bz$. Take a finite family of sets $A_1,\dots,A_n$ in $\bz$ such that $\frac1pA_i$ is a spectrum for $\Gamma$.

Now pick $0=r_1<r_2<\dots<r_n<r_{n+1}=\frac1p$ with the property that $r_i-r_j\not\in\mathbb Q$ unless $i=j$ or $\{i,j\}=\{1,n+1\}$. Define the set $\Omega$ as follows:
$$\Omega:=\bigcup_{i=1}^n\left([r_i,r_{i+1})+\frac1pA_i\right).$$

For every $x\in[0,\frac1p)$, the sets $\Omega_x$ defined in Proposition \ref{pr4.19} are among the sets $A_i$. Using Proposition \ref{pr4.8}, since $\Gamma$ is a spectrum for all the sets $\frac1p\Omega_x$, we get that $\Omega$ is spectral with spectrum $\Gamma+p\bz$. By hypothesis $\Omega$ tiles $\br$ with some set $\mathcal T$. We can assume that $0\in\mathcal T$. 

We prove first that $\mathcal T$ is contained in $\frac1p\bz$. Note that $|\Omega|=1$. By \cite[Theorem 1 and 2]{LaWa96}, we know that $\mathcal T$ is periodic with some integer period $s$ and $\mathcal T=\{0=t_0,t_1,\dots,t_{J}\}+s\bz$ with $t_i$ rational for all $i$.

We claim that, if $\cj{(\Omega+t)}\cap\cj{(\Omega+t')}\neq \ty$ then $t-t'\in\frac1p\bz$. Indeed, if this is the case, from the tiling property, we see that either $t=t'$ or two intervals, one in $\Omega+t$, one in $\Omega+t'$ have a common endpoint. So we have $r_{i+1}+\frac{s_{i}}{p}+t=r_j+\frac{s_j}{p}+t'$ for some $i,j\in\{1,\dots,n\}$, $s_i\in A_i, s_j\in A_j$. But, since $t,t'\in\mathbb Q$, we get that $r_{i+1}-r_j$ is in $\mathbb Q$. So $i+1=j$ or $\{i+1,j\}=\{1,n+1\}$. In both cases $r_{i+1}-r_j\in\frac1p\bz$. Then $t-t'\in\frac1p\bz$. 

If $J>0$ then $\Omega+s\bz$ does not exhaust $\br$, so there exists an element $0\neq t_{j_1}\in\{t_0,\dots,t_J\}$ and $s_0,s_1\in\bz$ such that $\cj{(\Omega+s_0)}\cap\cj{(\Omega+s_1+t_{j_1})}\neq\ty$, and hence we obtain $t_{j_1}\in\frac1p\bz$. Making one more step in a similar manner, if $J>1$ then $\Omega+(\{0,t_{j_1}\}+s\bz)$ does not exhaust $\br$, so there exist a $t_{j_2}$ and $s_2$ such that $\cj{(\Omega+s_2+t_{j_2})}$ intersects nontrivially either a set of the form $\cj{(\Omega+s_0')}$ or a set of the form $\cj{(\Omega+s_1'+t_{j_1})}$. In both cases we conclude that $t_{j_2}\in\frac1p\bz$. Continuing this way we conclude that $\{0=t_0,t_1,\dots,t_J\}\subset \frac1p\bz$, and hence $\mathcal T\subset\frac1p\bz$.

%Next, consider the set 
%$$\mathcal R:=\left\{x\in\br : \mbox{ There exist }t_1=0,\dots, t_N\in\mathcal T\mbox{ such that }\cj{(\Omega+t_i)}\cap\cj{(\Omega+t_{i+1})}\mbox{ for all $i$, and }x\in\cj{(\Omega+t_N)}\right\}.$$
%We claim that $\mathcal R$ is both open and closed. First, we check it is open: if $x\in\mathcal R$, then let $t_1,\dots,t_N$ as before, $x\in\cj{(\Omega+t_N)}$. If $x$ is interior to $\Omega+t_N$ then the entire $\Omega+t_N$ is contained in $\mathcal R$. If $x$ is a boundary point for $\Omega+t_N$, then it appears also as a boundary point for some other $\Omega+t_{N+1}$. The union $\cj{(\Omega+t_N)}\cup\cj{(\Omega+t_{N+1})}$ contains an open neighborhood for $x$ and is contained in $\mathcal R$. To see that $\mathcal R$ is closed one can use a similar discussion. Clearly $\Omega$ is contained in $\mathcal R$ so $\mathcal R$ is non-empty. The fact that $\br$ is connected implies that $\mathcal R=\br$. But this implies that for every $t\in\mathcal T$, there exist $t_0=0,\dots, t_N=t$ such that $\cj{(\Omega+t_i)}\cap\cj{(\Omega+t_{i+1})}\neq \ty$ for all $i$. With the previous fact, we get that $t_{i+1}-t_i\in\frac1p\bz$ and, by induction $t\in\frac1p\bz$.

We prove now that $\frac1p A_i$ tiles $\frac1p\bz$ by $\mathcal T$, for all $i$, which proves (i). 

Let $k\in\bz$ and pick $y\in (r_i,r_{i+1})$ arbitrarily. Then there exists a unique $t\in\mathcal T$ such that $y+\frac kp\in\Omega+t$. So there exists $1\leq j\leq n$, $y'\in[r_j,r_{j+1})$ and $s_j\in A_j$ such that 
$y+\frac kp=y'+\frac{s_j}p+t$. This implies $y-y'\in\frac1p \bz$ which means that $y=y'$ and $i=j$. So $\frac kp=\frac{s_j}p+t$. 

Now, assume $\frac ap+t=\frac{a'}p+t'$ for some $a,a'\in A_i$ and $t,t'\in \mathcal T$. Then 
$$(\Omega+t)\cap (\Omega+t')\supset([r_i,r_{i+1})+\frac ap+t)\cap ([r_i,r_{i+1})+\frac{a'}p+t')=[r_i,r_{i+1})+\frac ap+t.$$
This implies that $t=t'$ and $a=a'$. So $\frac1p A_i$ tiles $\frac1p\bz$ by $\mathcal T$.  

\end{proof}

\begin{proof}[Proof of Theorem \ref{th4.13}]
That (iii) implies (i) follows from Theorem \ref{th4.8}. To see that (i) implies (iii), we show that the Universal Tiling Conjecture is true under these assumptions. The proof is similar to the one for Theorem \ref{th4.8}. Let $p$, $\Gamma$, $A_1,\dots, A_m$ as in Conjecture \ref{con4.8}.  

Now pick $0=r_1<r_2<\dots<r_{m+1}=\frac1p$ some rational points and define 
$$\Omega=\bigcup _{i=1}^m\left((r_i,r_{i+1})+\frac1pA_i\right).$$

For every $x\in[0,\frac1p)$, the sets $\Omega_x$ defined in Proposition \ref{pr4.19} are among the sets $A_i$. Using Proposition \ref{pr4.8}, since $\Gamma$ is a spectrum for all the sets $\frac1p\Omega_x$, we get that $\Omega$ is spectral with spectrum $\Gamma+p\bz$. By hypothesis $\Omega$ tiles $\br$ with some set $\mathcal T$ contained in $\frac1p\bz$. We can assume $0\in\mathcal T$. We prove that every $A_i$ tiles $\bz$ by $p\mathcal T$. 

Take $x\in(r_i,r_{i+1})$ and $k\in\bz$. Then there exists $t\in\mathcal T$ such that $x+\frac kp\in\Omega+t$, so there exist $j\in\{1,\dots,m\}$, $y\in (r_j,r_{j+1})$ and $s\in A_j$ such that $x+\frac kp=y+\frac sp+t$. This implies that $x-y\in\frac1p\bz$ and since $x,y\in[0,\frac1p)$, this means that $x=y$. Hence $i=j$ and $k=s+pt$. 

Now assume $s+pt=s'+pt'$ for some $s,s'\in A_i$ and $t,t'\in\mathcal T$. Then $(\Omega+t)\cap (\Omega+t')$ contains $(r_i,r_{i+1})+\frac sp+t=(r_i,r_{i+1})+\frac{s'}p+t'$. So $t=t'$ and $s=s'$. Thus, $A_i$ tiles $\bz$ by $p\mathcal T$. 

Assume (i) and take $\Omega$ as in (ii). Then $\frac1N\Omega$ has rational endpoints and measure 1. If $\Lambda$ is a spectrum for $\Omega$ with period $\frac rN$, then $N\Lambda$ is a spectrum for $\frac1N\Omega$ with period $r$. By assumption, $\frac1N\Omega$ tiles $\br$ by a subset of $\frac1r\bz$. Then $\Omega$ tiles $\br$ by a subset $\frac Nr\bz$.

Assume (ii) and take $\Omega$ as in (i). Let $N$ be common denominator of all $\alpha_i,\beta_i$. Then $N\Omega$ has integer endpoints and measure $N$. If $\Lambda$ is a spectrum for $\Omega$ of period $r$, then $\frac1N\Lambda$ is a spectrum for $N\Omega$, with period $\frac rN$ (which is a multiple of the minimal period $\frac{r'}N$). By assumption, $N\Omega$ tiles $\br$ by a subset of $\frac N{r'}\bz\subset\frac Nr\bz$. Then $\Omega$ tiles $\br$ by a subset of $\frac1r\bz$.
\end{proof}

\begin{acknowledgements}
This work was done while the first named author (PJ) was visiting the University of Central Florida. We are grateful to the UCF-Math Department for hospitality and support. The authors are pleased to thank Professors Deguang Han, Steen Pedersen, Qiyu Sun and Feng Tian for helpful conversations. PJ was supported in part by the National Science Foundation, via a Univ of Iowa VIGRE grant. We thank the anonymous referee for the carefully reading the manuscript and for his/her suggestions that improved the paper significantly. This work was partially supported by a grant from the Simons Foundation (\#228539 to Dorin Dutkay).
\end{acknowledgements}
\bibliographystyle{alpha}
\bibliography{eframes}

\end{document}